\documentclass[12pt]{article}
\usepackage{amsmath, amsthm}
\theoremstyle{theorem}

\newtheorem{thm}{Theorem}[section]

\theoremstyle{definition}
\newtheorem{defn}[thm]{Definition}

\newtheorem{ex}[thm]{Example}
\theoremstyle{remark}

\numberwithin{equation}{section}

\begin{document}

\title{ On the Generalized Arcsine Probability Distribution with Bounded Support}%
\author{Rami AlAhmad\\Dept. of Mathematics, Yarmouk University\\ Irbid, JORDAN -- 21163\\e-mail: rami\_thenat@yu.edu.jo}
\date{\today}%

\maketitle

\begin{abstract}
Properties of  the beta functions are investigated. We define the generalized arcsine probability distribution with bounded support. The properties of the beta functions prove some results for this distribution.
\end{abstract}
\section{Introduction}The arcsine distribution plays main role in many fields. For example, the arcsine laws are a collection of results for one-dimensional random walks and Brownian motion (the Wiener process, see~\cite{Morters}). All three laws relate path properties of the Wiener process to the arcsine distribution.In particular, the third arcsine law states that the time at which a Wiener process achieves its maximum is arcsine distributed.\begin{defn}A random variable $X$ has the standard arcsine distribution if $X$ has probability density function $f$ given by
$$f(x)=\frac{1}{\sqrt{x(1-x)}},x\in(0,1).$$\end{defn}

\begin{defn}The beta function is defined as
 $$\beta(s,t) = \int_0^1 u^{s-1}(1-u)^{t-1}du, $$ for $\Re(s)>0$ and $\Re(t)>0$.
 \end{defn}\begin{thm}The central moments of a random variable $X$ with the standard arcsine distribution are $$\mu_k=\left\{
                                            \begin{array}{ll}
                                              0, & \hbox{k=2n-1;} \\
                                              \frac{(2n)!}{16^n(n!)^2} &\hbox{k=2n.}
                                            \end{array}
                                          \right. $$ where $n\in \textbf{N}.$

 \end{thm} This Theorem is a special case of Theorem~\ref{5}. 

\section{Additional properties of beta functions}
\begin{thm}\label{1} For $\Re(s)>0$, $\Re(t)>0$, and any two real numbers $r_1< r_2$. The beta function is give as
$$\beta(s,t)=\frac{1}{(r_2-r_1)^{s+t-1}}\int_{r_1}^{r_2}(x-r_1)^{s-1}(r_2-x)^{t-1}dx$$\end{thm}
\begin{proof}Using the substitution $w=\frac{x-r_1}{r_2-r_1}$ to get that
 \begin{align*}\int_{r_1}^{r_2}(x-r_1)^{s-1}(r_2-x)^{t-1}dx&=\int_{0}^{1}((r_2-r_1)w)^{s-1}(r_2-r_1-(r_2-r_1)w)^{t-1}(r_2-r_1)dw\\
 \\&=(r_2-r_1)^{s+t-1}\int_{0}^{1}w^{s-1}(1-w)^{t-1}dw\\&=(r_2-r_1)^{s+t-1}\beta(s,t).
 \end{align*}\end{proof}
 Moreover, the beta functions satisfy
\begin{thm}\label{2}For $r_1<\frac{r_1+r_2}{2}=\overline{r}< r_2$, $\Re(s)>0$, and $\Re(t)>0$
$$\int_{\overline{r}}^{r_2}(x-r_1)^{s-1}(x-\overline{r})^{t-1}(r_2-x)^{s-1}dx=\frac{1}{2}(\frac{r_2-r_1}{2})^{2s+t-2}\beta(s,\frac{t}{2}).$$
Moreover,
$$\int_{r_1}^{\overline{r}}(x-r_1)^{s-1}(\overline{r}-x)^{t-1}(r_2-x)^{s-1}dx=\frac{1}{2}(\frac{r_2-r_1}{2})^{2s+t-2}\beta(s,\frac{t}{2}).$$
In particular, the beta function can be defined as $$\beta(s,t)=\int_0^{1/2}x^{s-1}(\frac{1}{2}-x)^{2t-1}(1-x)^{s-1}=\int_{1/2}^1x^{s-1}(x-\frac{1}{2})^{2t-1}(1-x)^{s-1}.$$
\end{thm}
\begin{proof}Using the substitution $w=\left(\frac{2(x-\overline{r})}{r_2-r_1}\right)^2$ and the fact that $$\overline{r}-r_1=r_2-\overline{r}=\frac{r_2-r_1}{2}$$ to get that
 \begin{align*}&\int_{\overline{r}}^{r_2}(x-r_1)^{s-1}(x-\overline{r})^{t-1}(r_2-x)^{s-1}dx\\&=\int_{0}^{1}(\overline{r}-r_1+\frac{r_2-r_1}{2}\sqrt{w})^{s-1}(\frac{r_2-r_1}{2}\sqrt{w})^{t-1}(r_2-\overline{r}-\frac{r_2-r_1}{2}\sqrt{w})^{s-1}dw\\
&=\frac{1}{2}(\frac{r_2-r_1}{2})^{2s+t-2}\int_{0}^{1}(1-w)^{s-1}w^{\frac{t}{2}-1}dw=\frac{1}{2}(\frac{r_2-r_1}{2})^{2s+t-2}\beta(s,\frac{t}{2}).
 \end{align*}Now, following the same proof with $\sqrt{w}=\frac{2(\overline{r}-x)}{r_2-r_1}$ to prove the assertion $$\int_{r_1}^{\overline{r}}(x-r_1)^{s-1}(\overline{r}-x)^{t-1}(r_2-x)^{s-1}dx=\frac{1}{2}(\frac{r_2-r_1}{2})^{2s+t-2}\beta(s,\frac{t}{2}).$$
\end{proof}\begin{ex}using $r_1=1$ , $r_2=3$, $s=t=3/2$ and $\overline{r}=2$ to get that \begin{align*}\int_2^3\sqrt{6+6x^2-11x-x^3}dx&=\int_2^3\sqrt{(3-x)(x-2)(x-1)}dx\\&=\frac{1}{2}\beta(3/2, 3/4)\cong0.47925609389.\end{align*}

Also,  using $r_1=-1$ , $r_2=3$, $t=7$, $s=1/2$ and $\overline{r}=1$ to get that $$\int_{-1}^1\frac{(1-x)^6}{\sqrt{3+2x-x^2}}dx
=\int_{-1}^1(1-x)^{6}(3-x)^{-1/2}(1+x)^{-1/2}dx=\frac{1}{2}(2)^{6}\beta(1/2,7/2)=10\pi$$
\end{ex}

\subsection{The generalized arcsine probability distribution with bounded support}

Using Theorem~\ref{1}, we can easily show that  for $s,t >0$, the function $$f(x)=\frac{(x-r_1)^{s-1}(r_2-x)^{t-1}}{(r_2-r_1)^{s+t-1}\beta(s,t)}, r_1<x<r_2,$$ and zero elsewhere is probability density function. Moreover, this Theorem implies that the mean of the random variable $X$ is
\begin{align*}\mu=E(X)&=\int_{r_1}^{r_2}x\frac{(x-r_1)^{s-1}(r_2-x)^{t-1}}{(r_2-r_1)^{s+t-1}\beta(s,t)}dx\\&
=\int_{r_1}^{r_2}\frac{(x-r_1)^{s}(r_2-x)^{t-1}}{(r_2-r_1)^{s+t-1}\beta(s,t)}dx\\&+r_1\int_{r_1}^{r_2}\frac{(x-r_1)^{s-1}(r_2-x)^{t-1}}{(r_2-r_1)^{s+t-1}\beta(s,t)}dx
=r_1+\frac{s}{s+t}(r_2-r_1).\end{align*} Consider the following special case when $s=t$
\begin{defn} The random variable $X$ has is a generalized arcsine random variable if the probability density function is support$$f(x)=\frac{{((x-r_1)(r_2-x))}^{s-1}}{(r_2-r_1)^{2s-1}\beta(s,s)}, r_1<x<r_2.$$\end{defn} The above argument shows that this function is a valid density function. Moreover, it implies that in this case, when $s=t$, $\mu=\frac{r_2+r_1}{2}=\overline{r}$.

 As a special case, if $s=1/2$ then this distribution is called the arcsine probability distribution with bounded support ( see~\cite{Rogo}).  $s=1/2$ and $r_1=0$, $r_2=1$. Then $X$ has the standard arcsine distribution.
\begin{thm}\label{5}The central moments of a random variable $X$ with the standard arcsine distribution are $$\mu_k=\frac{((-1)^k+1)(r_2-r_1)^k\beta(s,\frac{k+1}{2})}{2^{2s+k}\beta(s,s)}=\left\{
                                            \begin{array}{ll}
                                              0, & \hbox{k=2n-1;} \\
                                              \frac{(r_2-r_1)^{2n}\beta(s,n+\frac{1}{2})}{2^{2s+2n-1}\beta(s,s)} &\hbox{k=2n.}
                                            \end{array}
                                          \right. $$ where $n\in \textbf{N}.$

 \end{thm}
\begin{proof}
\begin{align*}\mu_k=E(X-\mu)^k&=\int_{r_1}^{r_2}(x-\mu)^k\frac{{((x-r_1)(r_2-x))}^{s-1}}{(r_2-r_1)^{2s-1}\beta(s,s)}dx\\
&=\int_{r_1}^{r_2}(x-\overline{r})^k\frac{{((x-r_1)(r_2-x))}^{s-1}}{(r_2-r_1)^{2s-1}\beta(s,s)}dx\\
&=\frac{1}{(r_2-r_1)^{2s-1}\beta(s,s)}(\int_{r_1}^{\overline{r}}(x-\overline{r})^k{((x-r_1)(r_2-x))}^{s-1}dx\\
&+\int_{\overline{r}}^{r_2}(x-\overline{r})^k{((x-r_1)(r_2-x))}^{s-1}dx).\end{align*}
Using Theorem~\ref{2}, \begin{align*}\mu_k&=\frac{1}{(r_2-r_1)^{2s-1}\beta(s,s)}((-1)^k\int_{r_1}^{\overline{r}}(\overline{r}-x)^k{((x-r_1)(r_2-x))}^{s-1}dx\\
&+\int_{\overline{r}}^{r_2}(x-\overline{r})^k{((x-r_1)(r_2-x))}^{s-1}dx)\\&=\frac{((-1)^k+1)(r_2-r_1)^k\beta(s,\frac{k+1}{2})}{2^{2s+k}\beta(s,s)}.\end{align*}
In other words, for $n\in\textbf{N}$, we have $\mu_{2n-1}=0$ and $\mu_{2n}=
\frac{(r_2-r_1)^{2n}\beta(s,n+\frac{1}{2})}{2^{2s+2n-1}\beta(s,s)}$.

Consequently, If $X$ has arcsine probability distribution with bounded support then $\mu_{2n-1}=0$ and  using the fact that $\Gamma(n+\frac{1}{2})= \frac{(2n)!\sqrt{\pi}}{4^nn!} $ $$\mu_{2n}=\frac{(r_2-r_1)^{2n}\beta(\frac{1}{2},n+\frac{1}{2})}{4^{n}\beta(\frac{1}{2},\frac{1}{2})}
=\frac{(r_2-r_1)^{2n}(2n)!}{16^n(n!)^2}.$$\end{proof}
For the standard arcsine distribution, the central moments are $\mu_{2n-1}=0$ and $\mu_{2n}=\frac{(2n)!}{16^n(n!)^2}.$

 \end{document}